\documentclass{article}

\usepackage[twoside,
paperwidth=210mm,
paperheight=297mm,
textheight=622pt,
textwidth=468pt,
centering,
headheight=50pt,
headsep=12pt,
footskip=18pt,
footnotesep=24pt plus 2pt minus 12pt,
columnsep=2pc]{geometry}

\usepackage[auth-sc]{authblk}

\usepackage{mathtools}

\allowdisplaybreaks

\usepackage{amssymb}
\usepackage[mathscr]{eucal}

\usepackage{mismath}

\usepackage{hyperref}

\usepackage{bm} 

\usepackage{graphicx}
\graphicspath{{figs/}}

\usepackage{tikz}
\usetikzlibrary{math}
\usetikzlibrary{calc}

\usepackage{subcaption}

\usepackage{multirow}

\usepackage{makecell}
\setcellgapes{2pt}

\usepackage{booktabs}

\usepackage{siunitx}
\usepackage{algorithm}
\usepackage{algpseudocode}

\usepackage{relsize}

\newcommand{\SubplotTag}[1]{\textbf{\small \textsf{#1}}}

\DeclareMathOperator{\Div}{div}

\DeclareMathOperator{\Image}{im}
\DeclareMathOperator{\Card}{card}

\DeclareMathOperator{\Supp}{supp}
\DeclareMathOperator{\Int}{int}
\DeclareMathOperator{\Diam}{diam}

\DeclarePairedDelimiter{\RoundBrackets}{(}{)}
\DeclarePairedDelimiter{\CurlyBrackets}{\{}{\}}
\DeclarePairedDelimiter{\SquareBrackets}{[}{]}

\newcommand{\mathdefault}[1][]{}

\usepackage[skins]{tcolorbox}
\newtcolorbox{justabox}[2][]{%
  enhanced,
  attach boxed title to top center={yshift=-3mm,yshifttext=-1mm},
  colframe=blue!75!black,
  colbacktitle=red!80!black,
  fonttitle=\bfseries,
  title=#2,#1
}

\usepackage{wasysym}

\usepackage{amsthm}

\usepackage[capitalise]{cleveref}

\newtheorem{theorem}{Theorem}[section]
\newtheorem{lemma}[theorem]{Lemma}

\theoremstyle{definition}
\newtheorem{definition}{Definition}

\crefname{assumption}{assumption}{assumptions}
\Crefname{assumption}{Assumption}{Assumptions}

\crefname{problem}{problem}{problems}
\Crefname{problem}{Problem}{Problems}
\crefname{equation}{}{}
\Crefname{equation}{}{}

\theoremstyle{remark}

\usepackage{lineno}

\title{Multiscale modeling for problems with high contrast
heterogeneous coefficients by the
CEM-GMsFEM}
\author[1]{Eric T. Chung\thanks{\href{mailto:eric.t.chung@cuhk.edu.hk}{eric.t.chung@cuhk.edu.hk}}}
\author[2]{Yalchin Efendiev}
\author[1]{Xingguang Jin}
\author[3]{Wing Tat Leung}
\author[1]{Changqing Ye\thanks{\href{mailto:changqingye@cuhk.edu.hk}{changqingye@cuhk.edu.hk}}}
\affil[1]{Department of Mathematics, The Chinese University of Hong Kong, Shatin, Hong Kong}
\affil[2]{Department of Mathematics, Texas A \& M University, College Station, TX 77843, USA}
\affil[3]{Department of Mathematics, City University of Hong Kong, Hong Kong}

\begin{document}
\maketitle
\begin{abstract}
  This {review paper} provides a comprehensive overview of the Constrained Energy Minimizing Generalized Multiscale Finite Element Method (CEM-GMsFEM) for solving elliptic PDEs characterized by highly heterogeneous, high-contrast coefficients.
We detail the construction of multiscale basis functions via spectral auxiliary spaces, combined with an oversampling strategy that enables localized computations and guarantees exponential error decay.
Rigorous error estimates are outlined for reference to confirm the method's optimal convergence and robustness. { Numerical simulations are provided to verify the exponential decay property of the multiscale basis functions.} Additionally, we discuss and comment several up-to-date applications of CEM-GMsFEMs.

\end{abstract}

\section{Introduction}
\label{sec:introduction}
Many practical problems motivate the investigation of partial differential equations (PDEs) with inhomogeneous coefficients.
For instance, applications such as Darcy's law in inhomogeneous or fractured media and elasticity systems in composite materials illustrate these challenges.
When the coefficients exhibit particular structures---such as periodicity or randomness---an extensive body of mathematical theory has been developed \cite{Bensoussan2011,DalMaso1993,Jikov1994,Pankov1997,Conca1997,Cioranescu1999,Cioranescu2008,Tartar2009,Shen2018,Armstrong2019}, which underpins multiscale modeling and simulation.
However, general inhomogeneous coefficients, commonly accompanied by high contrast channels, pose a long-standing challenge for traditional methods.
This difficulty arises for two main reasons: channelized structures necessitate the use of fine meshes that dramatically increase the number of degrees of freedom, and high contrast ratios deteriorate the convergence of solvers for the resulting linear algebraic systems.

The heterogeneous nature of coefficients in {model problems} 
motivates the use of multiscale computational methods. Pioneered by Hou and Wu in \cite{Hou1997}, the concept of embedding model information into finite element space construction, known as MsFEMs, has received considerable attention.
One advantage of MsFEMs is that the fine-scale heterogeneity of the coefficients need not be fully resolved by the computational mesh, although the practical implementation relies on a pair of nested meshes.
To mitigate the constraints imposed by boundary conditions during the construction of multiscale basis functions, the oversampling technique was introduced in \cite{Hou1997} and later shown to improve convergence rates \cite{Efendiev2000,Efendiev2009}.
However, when the coefficient fails to satisfy the scale-separation assumption, the accuracy of MsFEMs may deteriorate, as demonstrated by the convergence analyses in \cite{Hou1999,Ye2020,Ming2024}.
To overcome this limitation, Efendiev, Galvis, and Hou proposed the Generalized Multiscale Finite Element Method (GMsFEM) in \cite{Efendiev2013}.
GMsFEMs employ spectral decomposition to reduce the dimensionality of the online space and have been shown to perform well for high-contrast and channel-like coefficient profiles \cite{Chung2014a}.
The first construction of multiscale basis functions that achieve the best theoretical approximation for general $L^\infty$ coefficients was presented by M\aa{}lqvist and Peterseim in \cite{Maalqvist2014}.
This approach, known as the Localized Orthogonal Decomposition (LOD), makes use of quasi-interpolation operators to separate the solution into macroscopic and microscopic components \cite{Altmann2021,Maalqvist2021}.
The subsequent combination of GMsFEMs and LOD resulted in the development of the Constraint Energy Minimizing GMsFEM (CEM-GMsFEM) by Chung, Efendiev, and Leung in \cite{Chung2018}.
The innovation in CEM-GMsFEM lies in the replacement of the quasi-interpolation operators from LOD with element-wise eigenspace projections, along with a relaxed formulation of the energy minimization problems to construct multiscale bases, thus removing the need to solve saddle-point linear systems.

We emphasize that the contrast robustness provided by CEM-GMsFEMs is an essential attribute that should be ``golden rule'' for any multiscale method.
Focusing solely on the heterogeneity of the coefficients is, to some extent, unconvencing for practical applications.
For example, when solving a heterogeneous elliptic equation \( -\Div(\kappa \nabla u) = f \), one can construct a preconditioner using the homogeneous counterpart \( -\Delta u = f \), which ensures rapid convergence of iterative solvers as long as the contrast is moderate.
The point lies in that, efficient solvers for the Poisson equation---such as multigrid, FFT, and fast multipole methods \cite{Gholami2016,Ye2024b}, which typically scale almost linearly with the number of degrees of freedom, are highly optimized in mature software packages.
However, when the contrast ratio is significantly high, these fast solvers may fail to converge, and multiscale methods become a viable alternative.
In fact, several current contrast-robust preconditioners share conceptual similarities with CEM-GMsFEMs, particularly in their use of spectral spaces \cite{Ye2024c,Ye2025}.

The remainder of the paper is organized as follows.
In \cref{sec:anatomy}, we describe the essential modules of CEM-GMsFEMs and provide implementation details.
\Cref{sec:keys} presents the theoretical analysis, outlining the crucial techniques used in obtaining the error estimates.
In \cref{sec:gallery}, we showcase and comment recent applications. \cref{sec:num} demonstrates the exponential decay of the multiscale basis functions. Finally, \cref{sec:conclusion} offers concluding remarks and an outlook on future research directions.

\section{Anatomy of CEM-GMsFEMs}
\label{sec:anatomy}
In this section, we discuss the key components of CEM-GMsFEMs, with particular emphasis on implementation aspects.
The model problem considered is the standard heterogeneous elliptic equation with Dirichlet boundary conditions:
\begin{equation}
  \label{eq:model}
  \left\{
  \begin{aligned}
     & -\Div(\kappa \nabla u) = f, &  & \quad \text{in } \Omega,         \\
     & u = 0,                      &  & \quad \text{on } \partial\Omega,
  \end{aligned}
  \right.
\end{equation}
where $\kappa$ is the diffusion coefficient satisfying the conventional ellipticity condition, and $f$ denotes the source term.
The domain $\Omega$ is assumed to be a bounded Lipschitz domain in $\mathbb{R}^d$ with $d=2,3$, where standard Sobolev spaces are defined.

\subsection{Two-scale nested meshes}
We assume that two partitions $\mathcal{T}_H$ and $\mathcal{T}_h$ over the domain $\Omega$ are given, where $H$ and $h$ are the mesh sizes of $\mathcal{T}_H$ and $\mathcal{T}_h$, respectively.
Those sizes can be characterized by mathematical expressions as
\[
  H\coloneqq \max_{K_H\in\mathcal{T}_H} \Diam(K_H),\quad h\coloneqq \max_{K_h\in\mathcal{T}_h} \Diam(K_h).
\]
The two partitions are nested in the sense that each subdomain $K_H \in \mathcal{T}_H$ is refined into a collection of $K_h \in \mathcal{T}_h$.
At this stage, those concepts have not touched FEMs, and those partitions can be very general and flexible.
For example, $\mathcal{T}_H$ can be a structured mesh, while $\mathcal{T}_h$ can be unstructured.
Starting from a partition that fully resolves the heterogeneity, there are two common approaches to construct these nested partitions:
\begin{enumerate}
  \item The fine partition $\mathcal{T}_h$ is given, and the coarse partition $\mathcal{T}_H$ is generated by coarsening/aggregating the fine partition.
  \item The coarse partition $\mathcal{T}_H$ is provided, and the fine partition $\mathcal{T}_h$ is generated by refining the coarse partition.
\end{enumerate}
For elliptic problems, accuracy is typically satisfactory if the fine partition adequately captures the heterogeneity of the coefficients, making the first approach more commonly used.
However, for more challenging problems, such as high-frequency wave propagation, the oscillations in the solution can be significantly more intricate than the heterogeneity introduced by the coefficients.
In such cases, the second approach becomes equally important.

From the fine partition $\mathcal{T}_h$, we can exercise a particular numerical scheme, commonly FEMs, to solve the problem.
However, the high resolution of the fine partition can lead to prohibitively expensive computational costs.
The core idea of CEM-GMsFEMs is to address this issue by transferring the computational burden, particularly the solving of linear systems, to the coarse partition $\mathcal{T}_H$.
This is accomplished by constructing a set of basis functions on the coarse partition, which are then used to approximate the solution on the fine partition.
In practice, the fine partition serves two primary purposes: (1) generating reference solutions to evaluate the accuracy of the coarse-scale solutions, and (2) implement the necessary components for constructing the basis functions.

From the fine partition $\mathcal{T}_h$ with the certain numerical scheme, we can obtain the fine-scale linear algebraic system as
\[
  A_h u_h = f_h.
\]
The corresponding coarse-scale system can be represented as
\[
  A_H u_H = f_H,
\]
where $A_H$ is derived from multiscale method.
Typically, the operator $A_H$ takes the form as
\[
  A_H = R_{H,h} A_h P_{h,H},
\]
where the two matrices $R_{H,h}$ and $P_{h,H}$ encapsulate the information of multiscale basis functions, and $f_H = R_{H,h}f_h$.
Specifically, each column of $P_{h,H}$ is a multiscale (trial) basis function represented in the fine-scale, while each row of $R_{H,h}$ is a multiscale (test) basis function.
Commonly, the matrix $R_{H,h}$ is the transpose of $P_{h,H}$, while it is not true for non-Hermitian problems such as Helmholtz equations with first-order absorbing boundary conditions.
Once $u_H$ is solved, the fine-scale approximation of $u_h$ is obtained by projecting $u_H$ back to the fine-scale as
\[
  u_h \approx P_{h,H} u_H.
\]
In practice, the offline phase refers to constructing $R_{H,h}$ and $P_{h,H}$ in a \emph{sparse matrix} format, and the online phase refers to solving the coarse-scale system and projecting the solution back to the fine-scale.

\subsection{Spectral auxiliary spaces}
Building upon the two-scale nested meshes, we can aggregate the fine DoFs into $\Card \mathcal{T}_H$ groups as $\mathcal{I}=\cup_{j}\mathcal{I}_j$, i.e., each group corresponds to a coarse subdomain.
For classic Lagrange FEMs, those fine DoF groups can overlap with each other, due to the nodes that are shared by multiple coarse subdomains.
For certain cell-central discretizations, e.g., finite volume methods, the fine DoF groups can be disjoint, although the study of such discretizations is still in its infancy.

At the continuous level, a spectral auxiliary space corresponding to the $j$-th coarse subdomain is constructed by solving an eigenvalue problem.
The mathematical formulation is given by
\begin{equation}
  \label{eq:eigenvalue}
  a_j(\Phi_{j,k}, v) = \lambda_{j,k} s_j(\Phi_{j,k}, v),\quad \forall v \in H^1(K_j),
\end{equation}
where $\Phi_{j,k}$ and $\lambda_{j,k}$ denote the eigenfunction and eigenvalue, respectively.
It is important to note that the boundary conditions must be homogeneous Neumann boundary conditions, rather than zero Dirichlet boundary conditions.
The bilinear form $a_j(\cdot, \cdot)$ is defined as the energy norm of the problem, specifically, $a_j(v, w)=\int_{\scriptscriptstyle K_j} \kappa \nabla u \cdot \nabla v \di x$.
However, the construction of the bilinear form $s_j(\cdot, \cdot)$ is more involved, which is essentially a weighted $L^2$ inner product as $s_j(v, w)=\int_{\scriptscriptstyle K_j} \tilde{\kappa} u v \di x$.
One common choice for the weight $\tilde{\kappa} = \kappa C_\text{msh} H^{-2}$, where $C_\text{msh}$ is a universal constant that can be determined by the regularity of the coarse mesh.
We also denote $a_\omega(\cdot, \cdot)$ and $s_\omega(\cdot, \cdot)$ as the bilinear forms defined on the subdomain $\omega$, and drop $\omega$ if $\omega$ is the whole domain $\Omega$.
The proportionality of $\tilde{\kappa}$ to $\kappa$ is essential for achieving contrast robustness.
Additionally, the inclusion of the factor $H^{-2}$ ensures that the eigenvalues $\lambda_{j,k}$ are dimensionless, with their growth pattern predictable by the Weyl law for convex domains.
The core idea of the spectral auxiliary space is to construct a projection operator $\pi_j$ that satisfies the following properties:
\begin{align*}
  \norm{\pi_j u}_{s_j}^2   & \leq \norm{\pi_j u}_{s_j}^2                                               \\
  \norm{u-\pi_i u}_{s_j}^2 & \leq \lambda_{j,k+1}^{-1} \norm{u}_{a_j}^2, \quad \forall u \in H^1(K_j).
\end{align*}
This highlights the necessity of homogeneous Neumann boundary conditions, as the energy bilinear form $a_j(\cdot, \cdot)$ must be defined for all $H^1(K_j)$ functions, rather than being restricted to $H^1_0(K_j)$ functions.
An important remark is that those eigenfunctions cannot be directly used as the basis functions, as they may not meet the global continuity condition.
However, various approaches have demonstrated that these eigenfunctions can be effectively incorporated into multilevel preconditioner frameworks, serving as coarse spaces and exhibiting superior contrast robustness compared to traditional methods

In practice, solving \cref{eq:eigenvalue} is typically reduced to a generalized eigenvalue problem:
\[
  A_j \Phi_{j,k} = \lambda_{j,k} S_j \Phi_{j,k},
\]
where the matrices $A_j$ and $S_j$ are derived from the bilinear forms $a_j(\cdot, \cdot)$ and $s_j(\cdot, \cdot)$, respectively.
Since the focus is on the eigenfunctions, the constant $C_\text{msh}$ does not need to be determined with high carefulness.
It is important to note that $A_j$ is positive semi-definite, while $S_j$ is positive definite, and only a few leading eigenfunctions are typically required.
As a result, sparse eigenvalue solvers, such as ARPACK and SLEPc, are well-suited for this task.
These solvers rely on Krylov subspace methods, and during each iteration, the inversion of $S_j$ may be required to transform the problem into a standard eigenvalue problem.
For large-scale eigenvalue problems, the mass-lumping technique, which replaces the mass matrix $S_j$ with a diagonal matrix, which can be trivially inverted, can be a practical and efficient alternative.

\subsection{Oversampling}
The oversampling technique was first introduced in the pioneering MsFEMs article to alleviate stiffness caused by boundary conditions imposed on coarse subdomains.
In the original MsFEMs, the oversampled domain is typically constructed by expanding a coarse subdomain with a few layers of fine elements, and hence the width of the oversampling layers is usually less than $H$.
In contrast, the oversampling concept in CEM-GMsFEMs is somewhat different: the oversampled domain is defined based on the neighboring relationship among coarse subdomains, with additional coarse subdomains recursively incorporated.
Consequently, an oversampled domain consisting of $k$ layers has a width of $kH$.
While it is straightforward to identify the oversampled domain for structured meshes, constructing it for unstructured meshes requires graph-based operations to recursively search for neighboring subdomains.

{\cref{fig:grid} serves as an illustration of the nested meshes, and note that two oversampling regions $K_{j'}^2$ and $K_{j''}^2$ are colored in gray.} For any $K_j \in \mathcal{T}_H$, we denote its oversampled domain by $K_{j,l}$, where $l$ represents the number of oversampling layers.
Through the oversampling process, we establish the following recursive relation for $K_{j,l}$:
\[
  K_{j,l} = \Int\RoundBrackets*{\cup\CurlyBrackets*{K_H \in \mathcal{T}_H \mid \overline{K_H} \cap \overline{K_{j,l-1}} \neq \varnothing}}, \quad \text{with } K_{j,0} = K_j.
\]
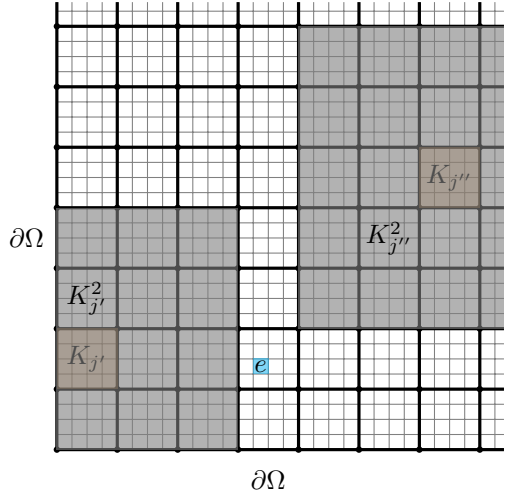
\begin{figure}[!ht]
  \centering
  \begin{tikzpicture}[scale=0.8]
    \draw[step=0.25, gray, thin] (0.0, 0.0) grid (7.4, 7.4);
    \draw[step=1.0, black, very thick] (0.0, -0.0) grid (7.4, 7.4);
    \foreach \x in {0,...,7}
    \foreach \y in {0,...,7}{
        \fill (1.0 * \x, 1.0 * \y) circle (1.5pt);
      }
    \fill[brown, opacity=0.4] (0.0, 1.0) rectangle (1.0, 2.0);
    \node at (0.5, 1.5) {$K_{j'}$};
    \fill[opacity=0.6, gray] (0.0, 0.0) rectangle (3.0, 4.0);
    \node at (0.5, 2.5) {$K_{j'}^2$};

    \fill[brown, opacity=0.4] (6.0, 4.0) rectangle (7.0, 5.0);
    \node at (6.5, 4.5) {$K_{j''}$};
    \fill[opacity=0.6, gray] (4.0, 2.0) rectangle (7.4, 7.0);
    \node at (5.5, 3.5) {$K_{j''}^2$};

    \fill [cyan, opacity=0.5] (3.25, 1.25) rectangle (3.5, 1.5);
    \node at (3.375, 1.375) {$e$};

    \node at (-0.5, 3.5) {$\partial \Omega$};
    \node at (3.5, -0.5) {$\partial \Omega$};
  \end{tikzpicture}
  \caption{
    Illustration of the two-scale nested meshes $\mathcal{T}_H$ and $\mathcal{T}_h$.
    A fine element $e$, two coarse elements $K_{j'}$ and $K_{j''}$, accompanied by their corresponding oversampling regions $K_{j'}^2$ and $K_{j''}^2$, are colored differently.
  }
  \label{fig:grid}
\end{figure}
As mentioned earlier, eigenfunctions are not suitable as basis functions because they lack continuity across coarse subdomain boundaries.
Hence, the pivotal idea behind oversampling is to extend and regularize these eigenfunctions by solving variational problems on the oversampled domain.
Specifically, for any eigenfunction $\Phi_{j,k}$, we can construct the corresponding multiscale basis function $\Psi_{j,k}$ using one of the following approaches:
\begin{itemize}
  \item \textbf{Constrained}: Find $\Psi_{j,k} \in H^1_0(K_{j,l})$ that minimizes
        \[
          \int_{K_{j, l}} \kappa \abs{\nabla v}^2 \, \di x,
        \]
        subject to
        \[
          s(v, \Phi_{j',k'}) = \delta_{j,j'} \delta_{k,k'}, \quad \forall\, j', k'.
        \]
  \item \textbf{Relaxed}: Find $\Psi_{j,k} \in H^1_0(K_{j,l})$ that minimizes
        \[
          \int_{K_{j, l}} \kappa \abs{\nabla v}^2 \, \di x + s(\pi v - \Phi_{j,k}, \pi v - \Phi_{j,k}).
        \]
\end{itemize}
Since the bilinear form $s_j(\cdot,\cdot)$ and the projection operator $\pi_j$ are defined piecewise without requiring continuity, we implicitly employ their global counterparts, $s(\cdot,\cdot)$ and $\pi$, over the entire domain $\Omega$.
Moreover, the relaxed formulation can be interpreted as a penalty method for the constrained approach by integrating the constraint into the objective function.
The final multiscale space is obtained by all $\Psi_{j,k}$, with each index $(j,k)$ corresponding to an eigenfunction.

In implementations, after constructing the oversampled domain, we build a mapping between the global DoFs and the local DoFs of the oversampled domain.
This mapping is used to correctly assemble the coarse-scale matrices $R_H$.
The discretizations for both the constrained and relaxed variational problems are carried out on the fine mesh $\mathcal{T}_h$.
For the constrained approach—which leads to a saddle point problem---we solve the linear system
\[
  \begin{bmatrix}
    Z & B^\intercal \\
    B & 0
  \end{bmatrix}
  \begin{bmatrix}
    \Phi \\
    \lambda
  \end{bmatrix}
  =
  \begin{bmatrix}
    0 \\
    b
  \end{bmatrix},
\]
where the block matrix $Z$ is derived from the bilinear form $a_{K_{j,l}}(\cdot,\cdot)$ and $B$ is constructed from the constraint $s_{K_{j,l}}(\pi\cdot,\pi\cdot)$.
For clarity, consider a coarse subdomain $K_j$: if all eigenvectors $\Phi_{j,\cdot}$ are arranged as the columns of a matrix $P_j$, then the projection operator $\pi_j$ can be expressed in the matrix form as $P_j P_j^\intercal S_j$.
If only the weights before the eigenvectors are needed, the projection simplifies to $P_j^\intercal S_j$.
Consequently, the matrix $B$ is formed as
\[
  B = \sum_{j'} P_{j'}^\intercal S_{j'},
\]
with the summation taken over all $j'$ corresponding to coarse subdomains contained in $K_{j,l}$ (with DoFs properly mapped), and the right-hand side vector $b$ is determined similarly.
For the relaxed approach, the linear system is much simpler:
\[
  (Z + C)\Phi = b,
\]
where
\[
  C = \sum_{j'} (P_{j'}^\intercal S_{j'})^\intercal P_{j'}^\intercal S_{j'}.
\]
This formulation involves significantly fewer DoFs than the constrained version, thus reducing computational cost.
However, since the matrix $P_{j'}$ is generally not sparse, the efficiency of solving the associated linear system may deteriorate as the number of DoFs increases.
In certain scenarios, this drawback can be alleviated.
For instance, if the discretization scheme permits non-overlapping local solvers (i.e., the DoF partition $\mathcal{I}=\cup_j\mathcal{I}_j$ is disjoint), the matrix $C$ can be rewritten as a block matrix with the structure $C = D D^\intercal$, where $D$ is a tall matrix.
In this case, the inversion of $Z+C$ can be performed using the Sherman--Morrison formula.
Specifically, by defining $z = D^\intercal \Phi$, we obtain
\[
  \Phi = -Z^{-1}D z +Z^{-1}b,
\]
which leads to the reduced system
\[
  (I + D^\intercal Z^{-1}D) z = D^\intercal Z^{-1}b,
\]
whose dimension is considerably smaller depending only on the total number of eigenfunctions on the oversampled domain.
Although this reduced system is dense, solving it is considerably less expensive than tackling the original system.
The price here is to factorize $Z$ first, which can be facilitated by sparse matrix factorizations.

\section{Keys in theoretical analysis}
\label{sec:keys}
The main theoretical result of CEM-GMsFEMs is an error estimate for the coarse-scale solution $u_H$, established under minimal assumptions on $\kappa$.
In particular, the error is expressed in terms of the coarse mesh size $H$ and the number of oversampling layers $l$, and it remains robust with respect to the contrast of $\kappa$, i.e., $\max_x \kappa(x)/\min_x \kappa(x)$.
The analysis presented below is for the relaxed formulation, but the constrained formulation can be treated similarly.

To facilitate the analysis, we introduce (and recall) the following notations:
\begin{itemize}
  \item $V^{\mathup{aux}}_j$ denotes the auxiliary space corresponding to the $j$-th coarse subdomain. When the index $j$ is dropped, $V^{\mathup{aux}}$ represents the direct sum of all such spaces.
  \item $\pi_j: L^2(K_j) \rightarrow L^2(K_j)$ is the projection operator onto the auxiliary space $V^{\mathup{aux}}_j$.
  \item $V_{j,l}\coloneqq H^1_0(K_{j,l})$, $a_{j,l}(\cdot,\cdot)\coloneqq a_{K_{j,l}}(\cdot,\cdot)$, and $s_{j,l}(\cdot,\cdot)\coloneqq s_{K_{j,l}}(\cdot,\cdot)$ are the shorthand notations for the function space and corresponding bilinear forms defined on the oversampled domain $K_{j,l}$.
  \item Notations for norms: $\norm{\cdot}_{a}^2 \coloneqq a(\cdot, \cdot)$, and similarly for $\norm{\cdot}_{s}$, $\norm{\cdot}_{a_{j,l}}$, and $\norm{\cdot}_{s_{j,l}}$.
  \item $V\coloneqq H^1_0(\Omega)$.
  \item For any $v\in L^2(K_j)$, the operator $\mathcal{G}_{j,l}:L^2(K_j) \rightarrow V_{j,l}$ is defined by the variational problem: find $\mathcal{G}_{j,l}v \in V_{j,l}$ such that
        \begin{equation} \label{eq:G-j-l}
          a_{j,l}(\mathcal{G}_{j,l}v, w) + s_{j,l}(\pi \mathcal{G}_{j,l}v, \pi w) = s_{j,l}(\pi v, \pi w), \quad \forall\, w \in V_{j,l}.
        \end{equation}
  \item For any $v\in L^2(K_j)$, the operator $\mathcal{G}_{j,\infty}:L^2(K_j) \rightarrow V$ is defined by the variational problem: find $\mathcal{G}_{j,\infty}v \in V$ such that
        \begin{equation} \label{eq:G-j-infty}
          a(\mathcal{G}_{j,\infty}v, w) + s(\pi \mathcal{G}_{j,\infty}v, \pi w) = s(\pi v, \pi w), \quad \forall\, w \in V.
        \end{equation}
  \item By extending functions in $L^2(K_j)$ to $L^2(\Omega)$, we define
        \[
          \mathcal{G}_{\infty}\coloneqq \sum_j \mathcal{G}_{j,\infty}\colon L^2(\Omega)\rightarrow V.
        \]
  \item By lifting the range of $\mathcal{G}_{j,l}$ to $V$ and extending functions from $L^2(K_j)$ to $L^2(\Omega)$, we define
        \[
          \mathcal{G}_l \coloneqq \sum_{j} \mathcal{G}_{j,l}\colon L^2(\Omega)\rightarrow V.
        \]
  \item The global multiscale space is given by $V_{H,\infty}=\Image(\mathcal{G}_{\infty})$, and the localized multiscale space by $V_{H,l}=\Image(\mathcal{G}_{l})$.
  \item Finally, define $\Lambda \coloneqq \min_{j} \lambda_{j,k+1}$.
\end{itemize}

\subsection{The global multiscale space is optimal}
The starting point is to verify the optimality of the global multiscale space $V_{H,\infty}$.
For elliptic problems, this result is typically straightforward because we can effectively exploit orthogonality.
Define
\[
  W = V \cap \ker \pi,
\]
where $\pi$ is the global projection operator.
First, we show that for any $v \in W$,
\[
  \norm{v}_{s}^2 = \sum_{j} \norm{v}_{s_j}^2 \leq \sum_{j} \lambda_{j,k+1}^{-1} \norm{v}_{a_j}^2 \leq \Lambda^{-1} \norm{v}_{a}^2.
\]
Moreover, by the definition of $\mathcal{G}_\infty$, for any $v \in V$ it holds that
\[
  a(\mathcal{G}_{\infty}v, w) + s(\pi \mathcal{G}_{\infty}v, \pi w) = s(\pi v, \pi w), \quad \forall\, w \in V.
\]
Testing this equation with any $w \in W$, the right-hand side vanishes, thereby establishing the orthogonality between $V_{H,\infty}$ and $W$.
Conversely, if for some $w\in V$ the relation $a(v,w) = 0$ holds for all $v \in V_{H,\infty}$, then it must be that $\pi w = 0$.
This conclusion is not immediately obvious and requires a careful argument: first, the set $\{\mathcal{G}_\infty \Phi_{j,k}\}$ is linearly independent; then, the set $\{\pi \mathcal{G}_\infty \Phi_{j,k} - \Phi_{j,k}\}$ is also linearly independent; finally, the condition
\[
  s(\pi v' - \pi \mathcal{G}_\infty v', \pi w) = 0\quad \forall v' \in L^2(\Omega)
\]
forces $\pi w = 0$.
Therefore, we can establish the following lemma:
\begin{lemma}
  It holds that
  \[
    V = V_{H,\infty} \oplus W.
  \]
\end{lemma}

Take $e = u - u_{H,\infty}^\mathup{ms}$, where \( u_{H,\infty}^\mathup{ms} \) is the multiscale solution derived from the global multiscale space.
By Galerkin orthogonality,
\[
  a(e,v)=0,\quad \forall\, v\in V_{H,\infty},
\]
which implies that \( e\in W \).
Using the variational formulation of \cref{eq:model}, we obtain
\[
  \norm{e}_a^2 = a(e,e) = \int_{\Omega} f e\, \di x \leq \norm{f}_{s^*} \norm{e}_s \leq \frac{1}{\sqrt{\Lambda}} \norm{f}_{s^*} \norm{e}_a,
\]
where \( s^* \) denotes the dual norm corresponding to the weighted \( L^2 \) inner product \( s(\cdot,\cdot) \).
Recalling that the weight \(\tilde{s}\) includes the factor \( H^{-2} \), we can deduce
\[
  \norm{f}_{s^*} \le H\,\frac{1}{\min_x \kappa(x)} \norm{f}_{L^2(\Omega)}.
\]
Since this estimate depends on the coefficient \(\kappa\), we retain the notation \( \norm{\cdot}_{s^*} \).
Note that the \( \bigO(H) \) estimate is optimal, as indicated by the Kolmogorov n-width theory, which characterizes the best approximation for solutions with \(L^2\) right-hand sides.
We thus conclude the error estimate for the global multiscale space:
\begin{theorem}
  \label{thm:global}
  Let \( u \) be the solution of \cref{eq:model} and \( u_{H,\infty}^\mathup{ms} \) be the multiscale solution obtained from the global multiscale space. Then, it holds that
  \[
    \norm{u - u_{H,\infty}^\mathup{ms}}_a \le \frac{1}{\sqrt{\Lambda}}\, \norm{f}_{s^*}.
  \]
\end{theorem}

\subsection{A three-step proof for localization}
In this section, we present an estimate for $\mathcal{G}_l - \mathcal{G}_\infty$ in terms of the number of oversampling layers $l$.
The proof is organized into three steps, with each step addressing a different aspect of the localization process.
A key technique in our analysis is the use of cutoff functions defined relative to the coarse mesh $\mathcal{T}_H$:

\begin{definition}\label{def:cutoff}
  For a coarse subdomain $K_j \in \mathcal{T}_H$, a cutoff function $\chi_j^{l,l'}$ with $l'>l$ satisfies the following properties:
  \begin{equation}
    \begin{aligned}
      \chi_j^{l,l'}(x) \equiv 1\quad    & \text{in}\ K_{j,l},                   \\
      \chi_j^{l,l'}(x) \equiv 0\quad    & \text{in}\ \Omega\setminus K_{j,l'},  \\
      0\leq \chi_j^{l,l'}(x)\leq 1\quad & \text{in}\ K_{j,l'}\setminus K_{j,l}.
    \end{aligned}
  \end{equation}
\end{definition}

The construction of such cutoff functions is straightforward, particularly when a coarse triangulation of the domain is available.
In our analysis, we frequently use the following estimate: for any $j$, $l$, and $l'$,
\[
  \|\nabla \chi_j^{l,l'}\|_{L^\infty(\Omega)} \leq \frac{C_{\chi}}{H\,|l'-l|},
\]
where $C_{\chi}$ is a constant that depends solely on the regularity of the coarse mesh $\mathcal{T}_H$.
Therefore, we can safely assume that
\[
  \kappa \abs{\nabla \chi_j^{l,l'}}^2 \leq \tilde{\kappa},
\]
for any cutoff function $\chi_j^{l,l'}$.

\begin{lemma}[The first step]
  \label{lem:step1}
  Let $l \geq 1$.
  There exists a positive constant $0<\theta<1$ such that
  \[
    \norm{\mathcal{G}_{j,\infty} v}_{a(\Omega \setminus K_{j,l})}^2+\norm{\pi \mathcal{G}_{j,\infty} v}_{s(\Omega\setminus K_{j,l})}^2 \leq \theta^l \CurlyBrackets*{\norm{\mathcal{G}_{j,\infty} v}_{a}^2+\norm{\pi \mathcal{G}_{j,\infty} v}_{s}^2},
  \]
  for any $v \in L^2(K_j)$,
  where $\theta = c_{*}/(c_{*}+1)$ and
  \[
    c_{*}(\Lambda)=\max_{x\in [0,\frac{\pi}{2}]} \RoundBrackets*{\cos(x)+\sin(x)}\RoundBrackets*{\frac{\cos(x)}{\sqrt{\Lambda}}+\sin(x)}.
  \]
\end{lemma}
This lemma indicates that \(\mathcal{G}_{j,\infty}v\) decays exponentially away from the coarse subdomain \(K_j\), which in turn establishes the viability of localization.
\begin{proof}
  Replacing the test function with $\RoundBrackets{1-\chi_j^{l-1,l}} \mathcal{G}_{j,\infty}$ in \cref{eq:G-j-infty}, recalling $1-\chi_j^{l-1,l}\equiv 0$ in $K_{j,l-1}$ and $1-\chi_j^{l-1,l}\equiv 1$ in $\Omega\setminus K_{j,l}$, we then obtain
  \begin{align*}
     & \quad \norm{\mathcal{G}_{j,\infty} v}_{a\RoundBrackets{\Omega \setminus K_{j,l}}}^2+\norm{\pi \mathcal{G}_{j,\infty} v}_{s\RoundBrackets{\Omega\setminus K_{j,l}}}^2                                                                                                                  \\
     & =\int_{K_{j,l}\setminus K_{j,l-1}}\RoundBrackets{\chi_j^{l-1,l}-1}\kappa\nabla \mathcal{G}_{j,\infty} v\cdot \nabla \mathcal{G}_{j,\infty} v\di x +\int_{K_{j,l}\setminus K_{j,l-1}} \kappa \mathcal{G}_{j,\infty} v \nabla \mathcal{G}_{j,\infty} v \cdot \nabla \chi_j^{l-1,l}\di x \\
     & \quad+\int_{K_{j,l}\setminus K_{j,l-1}} \tilde{\kappa} \pi \mathcal{G}_{j,\infty} v\cdot \pi \SquareBrackets{\RoundBrackets{\chi_j^{l-1,l}-1}\mathcal{G}_{j,\infty} v} \di x \coloneqq I_1+I_2+I_3.
  \end{align*}

  According to \cref{def:cutoff}, we have $\chi_j^{l-1,l}-1\leq 0$ in $K_{j,l}\setminus K_{j,l-1}$, which gives $I_1\leq 0$. By the property of cutoff functions, we can derive
  \[
    I_2 \leq \norm{\mathcal{G}_{j,\infty} v}_{a\RoundBrackets{K_{j,l}\setminus K_{j,l-1}}}\norm{\mathcal{G}_{j,\infty} v}_{s\RoundBrackets{K_{j,l}\setminus K_{j,l-1}}}.
  \]
  For $I_3$, applying the Cauchy--Schwarz inequality and eigenspace expansion estimates, we have
  \begin{align*}
    I_3 & \leq \norm{\pi \mathcal{G}_{j,\infty} v}_{s\RoundBrackets{K_{j,l}\setminus K_{j,l-1}}} \norm{\pi \SquareBrackets{\RoundBrackets{\chi_j^{l-1,l}-1}\mathcal{G}_{j,\infty} v}}_{s\RoundBrackets{K_{j,l}\setminus K_{j,l-1}}} \\
        & \leq \norm{\pi \mathcal{G}_{j,\infty} v}_{s\RoundBrackets{K_{j,l}\setminus K_{j,l-1}}} \norm{\RoundBrackets{\chi_j^{l-1,l}-1}\mathcal{G}_{j,\infty} v}_{s\RoundBrackets{K_{j,l}\setminus K_{j,l-1}}}                      \\
        & \leq \norm{\pi \mathcal{G}_{j,\infty} v}_{s\RoundBrackets{K_{j,l}\setminus K_{j,l-1}}} \norm{\mathcal{G}_{j,\infty} v}_{s\RoundBrackets{K_{j,l}\setminus K_{j,l-1}}}.
  \end{align*}
  Meanwhile, it also holds that $\norm{\mathcal{G}_{j,\infty} v}_{s\RoundBrackets{K_{j,l}\setminus K_{j,l-1}}}$ as
  \begin{align*}
    \norm{\mathcal{G}_{j,\infty} v}_{s\RoundBrackets{K_{j,l}\setminus K_{j,l-1}}}\leq & \norm{\mathcal{G}_{j,\infty} v-\pi \mathcal{G}_{j,\infty} v}_{s\RoundBrackets{K_{j,l}\setminus K_{j,l-1}}}+\norm{\pi \mathcal{G}_{j,\infty} v}_{s\RoundBrackets{K_{j,l}\setminus K_{j,l-1}}} \\
    \leq                                                                              & \frac{1}{\sqrt{\Lambda}} \norm{\mathcal{G}_{j,\infty} v}_{a\RoundBrackets{K_{j,l}\setminus K_{j,l-1}}}+\norm{\pi \mathcal{G}_{j,\infty} v}_{s\RoundBrackets{K_{j,l}\setminus K_{j,l-1}}}.
  \end{align*}
  Collecting all the estimates for $I_1$, $I_2$ and $I_3$, we arrive at
  \begin{align*}
         & \norm{\mathcal{G}_{j,\infty} v}_{a\RoundBrackets{\Omega \setminus K_{j,l}}}^2+\norm{\pi \mathcal{G}_{j,\infty} v}_{s\RoundBrackets{\Omega\setminus K_{j,l}}}^2                                                  \\
    \leq & \RoundBrackets*{\norm{\mathcal{G}_{j,\infty} v}_{a\RoundBrackets{K_{j,l}\setminus K_{j,l-1}}}+\norm{\pi \mathcal{G}_{j,\infty} v}_{s\RoundBrackets{K_{j,l}\setminus K_{j,l-1}}}} \times                         \\
         & \quad \RoundBrackets*{\frac{1}{\sqrt{\Lambda}} \norm{\mathcal{G}_{j,\infty} v}_{a\RoundBrackets{K_{j,l}\setminus K_{j,l-1}}}+\norm{\pi \mathcal{G}_{j,\infty} v}_{s\RoundBrackets{K_{j,l}\setminus K_{j,l-1}}}} \\
    \leq & c_{*}(\Lambda)\CurlyBrackets*{\norm{\mathcal{G}_{j,\infty} v}_{a\RoundBrackets{K_{j,l}\setminus K_{j,l-1}}}^2+\norm{\pi \mathcal{G}_{j,\infty} v}_{s\RoundBrackets{K_{j,l}\setminus K_{j,l-1}}}^2},
  \end{align*}
  where the constant $c_*(\Lambda)$ emerges by taking $x \in [0, \pi/2]$ such that
  \begin{align*}
    \cos(x) & = \frac{\norm{\mathcal{G}_{j,\infty} v}_{a\RoundBrackets{K_{j,l}\setminus K_{j,l-1}}}}{\sqrt{\norm{\mathcal{G}_{j,\infty} v}_{a\RoundBrackets{K_{j,l}\setminus K_{j,l-1}}}^2+\norm{\pi \mathcal{G}_{j,\infty} v}_{s\RoundBrackets{K_{j,l}\setminus K_{j,l-1}}}^2}},     \\
    \sin(x) & = \frac{\norm{\pi \mathcal{G}_{j,\infty} v}_{s\RoundBrackets{K_{j,l}\setminus K_{j,l-1}}}}{\sqrt{\norm{\mathcal{G}_{j,\infty} v}_{a\RoundBrackets{K_{j,l}\setminus K_{j,l-1}}}^2+\norm{\pi \mathcal{G}_{j,\infty} v}_{s\RoundBrackets{K_{j,l}\setminus K_{j,l-1}}}^2}}.
  \end{align*}
  This yields an iterative relation
  \begin{align*}
     & \quad \norm{\mathcal{G}_{j,\infty} v}_{a\RoundBrackets{\Omega \setminus K_{j,l-1}}}^2+\norm{\pi \mathcal{G}_{j,\infty} v}_{s\RoundBrackets{\Omega\setminus K_{j,l-1}}}^2                                                                                                                                                                \\
     & =\norm{\mathcal{G}_{j,\infty} v}_{a\RoundBrackets{\Omega \setminus K_{j,l}}}^2+\norm{\pi \mathcal{G}_{j,\infty} v}_{s\RoundBrackets{\Omega\setminus K_{j,l}}}^2 + \norm{\mathcal{G}_{j,\infty} v}_{a\RoundBrackets{K_{j,l} \setminus K_{j,l-1}}}^2+\norm{\pi \mathcal{G}_{j,\infty} v}_{s\RoundBrackets{K_{j,l} \setminus K_{j,l-1}}}^2 \\
     & \geq\RoundBrackets{1+\frac{1}{c_{*}}} {\CurlyBrackets*{\norm{\mathcal{G}_{j,\infty} v}_{a\RoundBrackets{\Omega \setminus K_{j,l}}}^2+\norm{\pi \mathcal{G}_{j,\infty} v}_{s\RoundBrackets{\Omega\setminus K_{j,l}}}^2}},
  \end{align*}
  and also finishes the proof.
\end{proof}

The second step provides an estimate on the difference between $\mathcal{G}_{j,\infty}$ and $\mathcal{G}_{j,l}$.
\begin{lemma}[The second step]
  \label{lem:step2}
  Keep the notations, then
  \[
    \norm{\RoundBrackets{\mathcal{G}_{j,\infty}-\mathcal{G}_{j,l}}v}_a^2+\norm{\pi \RoundBrackets{\mathcal{G}_{j,\infty}-\mathcal{G}_{j,l}}v}_s^2 \leq c_{\star}\theta^{l-1}\RoundBrackets{\norm{\mathcal{G}_{j,\infty} v}_{a}^2+\norm{\pi \mathcal{G}_{j,\infty} v}_{s}^2},
  \]
  where
  \[
    c_{\star}(\Lambda)=\max_{x\in [0,\frac{\pi}{2}]} \SquareBrackets*{\RoundBrackets{\frac{1}{\sqrt{\Lambda}}+1}\cos(x)+\sin(x)}^2+\RoundBrackets*{\frac{\cos(x)}{\sqrt{\Lambda}}+\sin(x)}^2.
  \]
\end{lemma}
\begin{proof}
  Let $z_j\coloneqq \RoundBrackets{\mathcal{G}_{j,\infty}-\mathcal{G}_{j,l}}v$, and decompose $z_j$ as
  \[
    z_j=\CurlyBrackets{\RoundBrackets{1-\chi_j^{l-1,l}}\mathcal{G}_{j,\infty} v}+\CurlyBrackets{\RoundBrackets{\chi_j^{l-1,l}-1}\mathcal{G}_{j,l} v + \chi_j^{l-1,l}z_j}\coloneqq z_j'+z_j''.
  \]
  Recalling the definition of $\chi_j^{l-1,l}$, we have $z_j''\in V_{j,l}$.
  Then, combining \cref{eq:G-j-infty,eq:G-j-l}, we obtain
  \[
    a(z_j,z_j'')+s(\pi z_j,\pi z_j'')=0.
  \]
  The techniques for estimating $a(z_j,z_j')+s(\pi z_j,\pi z_j'')$ are essentially the same as in the proof of \cref{lem:step1}:
  \[
    \norm{z_j}_a^2+\norm{\pi z_j}_s^2=a(z_j,z_j')+s(\pi z_j,\pi z_j') \leq \norm{z_j}_a \norm{z_j'}_a+\norm{\pi z_j}_s\norm{\pi z_j'}_s;
  \]
  \begin{align*}
    \norm{z_j'}_a & = \norm{\RoundBrackets{1-\chi_j^{l-1,l}}\mathcal{G}_{j,\infty} v}_a \leq \norm{\mathcal{G}_{j,\infty} v}_{a(\Omega\setminus K_{j,l-1})}+\norm{\mathcal{G}_{j,\infty} v}_{s(\Omega\setminus K_{j,l-1})} \\
                  & \leq \RoundBrackets{1+\frac{1}{\sqrt{\Lambda}}}\norm{\mathcal{G}_{j,\infty} v}_{a(\Omega\setminus K_{j,l-1})}+\norm{\pi \mathcal{G}_{j,\infty} v}_{s(\Omega\setminus K_{j,l-1})};
  \end{align*}
  \begin{align*}
    \norm{\pi z_j'}_s & =\norm{\pi\SquareBrackets{\RoundBrackets{1-\chi_j^{l-1,l}}\mathcal{G}_{j,\infty} v}}_s\leq \norm{\RoundBrackets{1-\chi_j^{l-1,l}}\mathcal{G}_{j,\infty} v}_s \leq \norm{\mathcal{G}_{j,\infty} v}_{s(\Omega \setminus K_{j,l-1})} \\
                      & \leq \frac{1}{\sqrt{\Lambda}}\norm{\mathcal{G}_{j,\infty} v}_{a(\Omega\setminus K_{j,l-1})}+\norm{\pi \mathcal{G}_{j,\infty} v}_{s(\Omega\setminus K_{j,l-1})}.
  \end{align*}
  We then obtain
  \begin{align*}
     & \quad \norm{z_j}_a^2+\norm{\pi z_j}_s^2=a(z_j,z_j')+s(\pi z_j,\pi z_j')                                                                                                                                                                              \\
     & \leq \CurlyBrackets*{\norm{z_j'}_a^2+\norm{\pi z_j'}_s^2}^{1/2}\CurlyBrackets*{\norm{z_j}_a^2+\norm{\pi z_j}_s^2}^{1/2}                                                                                                                              \\
     & \leq \CurlyBrackets*{ c_{\star}\RoundBrackets{\norm{\mathcal{G}_{j,\infty} v}_{a(\Omega\setminus K_{j,l-1})}^2+\norm{\pi \mathcal{G}_{j,\infty} v}_{s(\Omega\setminus K_{j,l-1})}^2}}^{1/2}\CurlyBrackets*{\norm{z_j}_a^2+\norm{\pi z_j}_s^2}^{1/2},
  \end{align*}
  where the constant $c_{\star}$ is obtained by choosing $x\in [0, \pi/2]$ such that
  \begin{align*}
    \cos(x) & = \frac{\norm{\mathcal{G}_{j,\infty} v}_{a\RoundBrackets{\Omega \setminus K_{j,l-1}}}}{\sqrt{\norm{\mathcal{G}_{j,\infty} v}_{a\RoundBrackets{\Omega \setminus K_{j,l-1}}}^2+\norm{\pi \mathcal{G}_{j,\infty} v}_{s\RoundBrackets{\Omega \setminus K_{j,l-1}}}^2}},     \\
    \sin(x) & = \frac{\norm{\pi \mathcal{G}_{j,\infty} v}_{s\RoundBrackets{\Omega \setminus K_{j,l-1}}}}{\sqrt{\norm{\mathcal{G}_{j,\infty} v}_{a\RoundBrackets{\Omega \setminus K_{j,l-1}}}^2+\norm{\pi \mathcal{G}_{j,\infty} v}_{s\RoundBrackets{\Omega \setminus K_{j,l-1}}}^2}}.
  \end{align*}
  Combining the result in \cref{lem:step1}, we hence complete this proof.
\end{proof}

The third step relies on the capacity assumption: There exists a positive constant $C_\mathup{ol}$ such that for all $K_j\in \mathcal{T}^H$ and $l > 0$,
\[
  \Card\CurlyBrackets{K \in \mathcal{T}^H\colon K \subset K_i^m} \leq C_\mathup{ol} l^d.
\]
This condition implies that the coarse partition is not overly irregular, a property that naturally holds for regular triangulations.

\begin{lemma}[The third step]
  \label{lem:step3}
  Keep the notations, then
  \[
    \norm{\RoundBrackets{\mathcal{G}_\infty-\mathcal{G}_{l}}v}_a^2+\norm{\pi \RoundBrackets{\mathcal{G}_\infty-\mathcal{G}_{l}}v}_s^2\leq c_\star^2 C_{\mathup{ol}} \theta^{l-1}(l+1)^d \norm{\pi v}_{s}^2,
  \]
  for any $v \in L^2(\Omega)$.
\end{lemma}
\begin{proof}
  Still take $z_j \coloneqq \RoundBrackets{\mathcal{G}_{j,\infty}-\mathcal{G}_{j,l}}v|_{K_j}$ and $z=\sum_{j} z_j$, and decompose $z$ as
  \[
    z=\CurlyBrackets{\RoundBrackets{1-\chi_j^{l,l+1}}z}+\CurlyBrackets{\chi_j^{l,l+1}z}\coloneqq z'+z''.
  \]
  Noting that $\Supp\RoundBrackets{z'} \subset \Omega\setminus K_i^{m}$, $\Supp \RoundBrackets{\pi z'} \subset \Omega\setminus K_i^{m}$, $\Supp\RoundBrackets{\mathcal{G}_{j,l} v} \subset \overline{\RoundBrackets{K_i^m}}$ and $\Supp\RoundBrackets{\pi \mathcal{G}_{j,l} v} \subset \overline{\RoundBrackets{K_i^m}}$, we have
  \[
    a(\mathcal{G}_{j,l} v, z')+s(\pi\mathcal{G}_{j,l} v, \pi z')=0
  \]
  and
  \[
    a(\mathcal{G}_{j,\infty} v, z')+s(\pi\mathcal{G}_{j,\infty} v, \pi z')=0,
  \]
  which leads to
  \[
    a(z_j,z')+s(\pi z_j, \pi z')=0.
  \]
  Use similar techniques in the proof of \cref{lem:step2}:
  \[
    a(z_j, z)+s(\pi z_j, \pi z)=a(z_j, z'')+s(\pi z_j, \pi z'') \leq \norm{z_j}_a\norm{z''}_a+\norm{\pi z_j}_s\norm{\pi z''}_s;
  \]
  \begin{align*}
    \norm{z''}_a & =\norm{\chi_j^{l,l+1}z}_a\leq \norm{z}_{a\RoundBrackets{K_{j,l+1}}}+\norm{z}_{s\RoundBrackets{K_{j,l+1}}}                       \\
                 & \leq \RoundBrackets{1+\frac{1}{\sqrt{\Lambda}}}\norm{z}_{a\RoundBrackets{K_{j,l+1}}}+\norm{\pi z}_{s\RoundBrackets{K_{j,l+1}}};
  \end{align*}
  \begin{align*}
    \norm{\pi z''}_s & = \norm{\pi \RoundBrackets{\chi_j^{l,l+1}z}}_s \leq \norm{\chi_j^{l,l+1}z}_s \leq \norm{z}_{K_{j,l+1}}        \\
                     & \leq \frac{1}{\sqrt{\Lambda}}\norm{z}_{a\RoundBrackets{K_{j,l+1}}}+\norm{\pi z}_{s\RoundBrackets{K_{j,l+1}}}.
  \end{align*}
  We hence obtain
  \[
    a(z_j,z)+s(\pi z_j, \pi z)\leq \CurlyBrackets*{c_\star\RoundBrackets{\norm{z}_{a\RoundBrackets{K_{j,l+1}}}^2+\norm{\pi z}_{s\RoundBrackets{K_{j,l+1}}}^2}}^{1/2}\CurlyBrackets*{\norm{z_j}_a^2+\norm{\pi z_j}_s^2}^{1/2}.
  \]
  Recalling the definition of $C_\mathup{ol}$, it is easy to show
  \[
    \sum_j \norm{z}_{a\RoundBrackets{K_{j,l+1}}}^2+\norm{\pi z}_{s\RoundBrackets{K_{j,l+1}}}^2\leq C_\mathup{ol}\RoundBrackets{m+1}^d\RoundBrackets{\norm{z}_a^2+\norm{\pi z}_s^2}.
  \]
  Then by the Cauchy--Schwarz inequality, we get
  \begin{align*}
    \norm{z}_a^2+\norm{\pi z}_s^2 & =\sum_j a(z_j,z)+s(\pi z_j,\pi z)                                                                                                                               \\
                                  & \leq \CurlyBrackets*{c_\star C_\mathup{ol} \RoundBrackets{\norm{z}_a^2+\norm{\pi z}_s^2}}^{1/2}\CurlyBrackets*{\sum_j \norm{z_j}_a^2+\norm{\pi z_j}_s^2}^{1/2}.
  \end{align*}
  Combining \cref{lem:step2} and
  \[
    \norm{\mathcal{G}_{j,\infty} v|_{K_j}}_a^2+\norm{\pi \mathcal{G}_{j,\infty} v|_{K_j}}_s^2=s(\pi v|_{K_j}, \pi \mathcal{G}_{j,\infty}v|_{K_j}),
  \]
  which leads to
  \[
    \norm{\mathcal{G}_{j,\infty} v|_{K_j}}_a^2+\norm{\pi \mathcal{G}_{j,\infty} v|_{K_j}}_s^2 \leq \norm{\pi v|_{K_j}}_s^2.
  \]
  We hence complete the proof.
\end{proof}

\subsection{The error estimate for the localized multiscale space}
We are now ready to establish the error estimate for the localized multiscale space.
Let $u_{H,l}^\mathup{ms}$ be the multiscale solution derived from the localized multiscale space $V_{H,l}$.
From the Galerkin orthogonality, we have $\norm{u-u_{H,l}^\mathup{ms}}_a \leq \norm{u-v_{H,l}}_a$ for any $v_{H,l} \in V_{H,l}$.
According to the definition for $V_{H,l}$, we can take $v_{H,l}=\mathcal{G}_{l}v$, where $v$ will be chosen later.
Then, take a splitting as
\[
  \norm{u-v_{H,l}}_a \leq \norm{u-u_{H,\infty}^\mathup{ms}}_a + \norm{u_{H,\infty}^\mathup{ms}-v_{H,l}}_a,
\]
and the first term is already estimated in \cref{thm:global}.
Again, we know that $u_{H,\infty}^\mathup{ms}=\mathcal{G}_\infty w$, and thus a proper choice of $v$ as $w$ for $v_{H,l}$ can be made.
The second term can be estimated by \cref{lem:step3} as follows:
\[
  \norm{u_{H,\infty}^\mathup{ms}-v_{H,l}}_a=\norm{\mathcal{G}_{\infty} w-\mathcal{G}_{l} w}_a \leq c_\star \sqrt{C_\mathup{ol}}\theta^{(l-1)/2} (l+1)^{d/2}\norm{\pi w}_s.
\]
The remaining task is to bound $\norm{\pi w}_s$ with $\norm{f}_{s^*}$.
We need the following lemma, which is in some sense an inverse inequality for the auxiliary space $V_{H,\infty}$:
\begin{lemma}[Inverse inequality]
  \label{lem:interpolation}
  There exists a bounded map $\mathcal{Q}_H\colon L^2(\Omega) \rightarrow V$ and a positive constant $C_\mathup{inv}$ such that for all $v \in L^2(\Omega)$, it holds that $\pi \mathcal{Q}_H v=\pi v$ and $\norm{\mathcal{Q}_H v}_{\tilde{a}} \leq C_\mathup{inv} \norm{\pi v}_{s}$.
  Moreover, for each coarse element $K_j$, $\mathcal{Q}_Hv|_{K_j}$ depends only on the data of $v$ in $K_j$ and vanishes on $\partial K_j$.
\end{lemma}
This lemma actually says that it is permitting to assume $w \in V$.
Then, according to the variational formulation for $\mathcal{G}_\infty$, we have
\[
  \norm{\pi w}_s^2 = a(\mathcal{G}_\infty w, w) + s(\pi \mathcal{G}_\infty w, \pi w).
\]
Through the Cauchy--Schwarz inequality, and the property that $\norm{w}_a \leq C_\mathup{inv}\norm{\pi w}_s$, we can derive
\[
  \norm{\pi w}_s^2 \leq C(C_\mathup{inv})\CurlyBrackets*{\norm{\mathcal{G}_\infty w}_a^2 + \norm{\pi \mathcal{G}_\infty w}_s^2}.
\]
By the Poincar\'{e} inequality, the term $\norm{\pi \mathcal{G}_\infty w}_s$ can be bounded aspects
\[
  \norm{\pi \mathcal{G}_\infty w}_s \leq \norm{\mathcal{G}_\infty w}_s \leq C_\mathup{po} H^{-1} \norm{\mathcal{G}_\infty w}_a.
\]
Finally, for $\mathcal{G}_\infty w$, we can utilize the variational form for $u_{H,\infty}^\mathup{ms}$, which gives
\[
  \norm{u_{H,\infty}^\mathup{ms}}_a^2 \leq \norm{f}_{s^*} \norm{u_{H,\infty}^\mathup{ms}}_s \leq C_\mathup{po} H^{-1} \norm{f}_{s^*} \norm{u_{H,\infty}^\mathup{ms}}_a.
\]
Summarize all the above discussions, we can conclude the error estimate for the localized multiscale space as follows:
\begin{theorem}
  \label{thm:local}
  Let \( u \) be the solution of \cref{eq:model} and \( u_{H,l}^\mathup{ms} \) be the multiscale solution obtained from the localized multiscale space $V_{H,l}$. Then, it holds that
  \[
    \norm{u - u_{H,l}^\mathup{ms}}_a \leq C_* \RoundBrackets*{1+ H^{-2}(l+1)^{d/2}\theta^{(l-1)/2}} \norm{f}_{s^*},
  \]
  where the positive constant $C_*$ and $\theta$ is independent of $H$ and $l$ with $\theta < 1$.
\end{theorem}

To overcome the error resulting from the factor \(H^{-2}\), the oversampling layers should be chosen as \(\bigO(\abs{\log H})\).
From the analysis, it is evident that the constant \(C_*\) depends on \(C_\mathup{inv}\), \(C_\mathup{po}\), and \(\Lambda\).
Under certain assumptions, \(\Lambda\) can be shown to be independent of the contrasts; however, the influence of contrasts on \(C_\mathup{inv}\) and \(C_\mathup{po}\) remains unclear.
Since the weighted \(L^2\) norm is used in constructing these two constants, we conjecture that similar independence may hold for them as well.
Moreover, the critical decay rate, provided by \(\theta\), is independent of the contrasts.

\section{Application gallery}
\label{sec:gallery}
In this section, we present several applications of CEM-GMsFEMs from the community.
We categorize these publications into several groups, although the classification is not strictly orthogonal. Special emphasis is placed on modifications to the original CEM-GMsFEM framework.

\subsection{Elasticity problems}
One obvious extension of CEM-GMsFEMs is to consider elasticity. {The classical linear elasticity problem is given by
\begin{equation}\label{eq:ela}
- \nabla \cdot \sigma(u) = f \quad \text{in } D\subset\mathbb{R}^{d},
\end{equation}
with homogeneous Dirichlet boundary condition $u = 0$ on $\partial D$, where $u$ is the displacement field and
\[
\sigma(u) = 2\mu \, \varepsilon(u) + \lambda (\nabla \cdot u) \mathcal{I},
\qquad
\varepsilon(u) = \frac{1}{2} \left( \nabla u + (\nabla u)^T \right).
\]
Here $\lambda>0$ and $\mu>0$ are heterogeneous Lamé parameters, $\mathcal{I}$ is the identity tensor.}
In \cite{Fu2018}, the authors proposed a CEM-GMsFEM for elasticity problems \cref{eq:ela} , demonstrating contrast robustness regarding Young's modulus.
For elasticity problems, the left-hand bilinear form in eigenvalue problems \cref{eq:eigenvalue} should be replaced by the elasticity energy form, while the right-hand side can be set as $(\lambda+2\mu) H^{-2}$, where \(\lambda\) and \(\mu\) are the Lam\'{e} parameters.
We emphasize that, according to the Korn inequality, the elasticity energy form is semi-definite with the kernel corresponding to the space of rigid body motions.
Therefore, for 2D problems, achieving contrast robustness requires at least $4$ eigenvectors, and for 3D problems, at least $6$ eigenvectors are needed. {The framework \cref{eq:ela} has also been extended to poroelasticity problems, where the displacement field $u$ is coupled with a scalar pressure field $p: [0,T]\times D\rightarrow\mathbb{R}$. 
A typical linear poroelastic system reads
\begin{equation}\label{eq:proela}
\begin{cases}
- \nabla \cdot \sigma(u) + \alpha \nabla p = 0,  &\quad \text{in } (0,T]\times D,\\
M^{-1} \partial_t p + \alpha \nabla \cdot \partial_t u - \nabla \cdot (\frac{\kappa}{\nu} \nabla p) = f &\quad \text{in } (0,T]\times D,
\end{cases}
\end{equation}
where $\alpha$ is the Biot coefficient, $M$ is the Biot modulus, $\nu$ is the fluid viscosity and $\kappa$ denotes heterogeneous permeability. }In \cite{Fu2019}, the authors considered poroelasticity problems \cref{eq:proela} with homogeneous Dirichlet boundary conditions.
In poroelasticity, a scalar field---pressure---is incorporated into the model, analogous to the thermal-elasticity model.
As the system is coupled, the authors proposed two different types of multiscale spaces: one for the pressure to capture conductivity heterogeneity and the other for the displacement to address mechanical heterogeneity.
Note that the construction of those multiscale spaces is decoupled, and the multiscale solution is obtained by solving the variational problem in these multiscale spaces.
The extension to nonlinear poroelasticity is also discussed in \cite{Fu2020}; due to the nonlinearity, the update of multiscale bases is required in each time step.

\subsection{Alternative discretizations}
Discretizations beyond the finite element methods are crucial when mass conservation is enforced or when incompressibility plays a role.
In \cite{Chung2018b}, the multiscale method is formulated using the mixed form of elliptic problems, meaning that a pair of $H(\Div,\Omega) \times L^2(\Omega)$ conforming multiscale spaces is constructed. {The authors consider a class of high-contrast flow problems in the following mixed form:
\begin{equation}\label{eq:mixed}
\text{div}\,u= f,\quad \kappa^{-1} u = - \nabla p,
\quad \text{in } D,
\end{equation}
subject to the homogeneous boundary condition $u \cdot \vec{n} = 0$ on $\partial D,$
where $\vec{n}$ is the outward unit normal vector on $\partial D$. The source function satisfies $\int_D f \, dx = 0.$}
For the pressure component {$p$ in \cref{eq:mixed}}, the multiscale space is essentially the eigenfunction space of \cref{eq:eigenvalue}, presented within the mixed formulation.
For the velocity component {$u$}, the multiscale space is obtained by solving the oversampled problem.
The mixed formulation offers the advantage that all degrees of freedom (DoFs) for eigenfunctions are disjoint, thus enabling the use of the Sherman--Morrison formula.
Moreover, in mixed formulations the homogeneous Neumann boundary condition is more natural for oversampling problems, as it results from the zero-extension operation of $H(\Div,\Omega)$ functions.
An adaptive strategy based on the mixed form is introduced in \cite{Chung2019a}, where, rather than remeshing the domain as in classical adaptive finite element methods, an enrichment of the multiscale bases through local residuals is proposed.
In \cite{Cheung2020}, CEM-GMsFEMs are extended to the IPDG framework.
In this case, the right-hand side of \cref{eq:eigenvalue} comes from the IPDG energy bilinear, and again, all DoFs on coarse subdomains are disjoint due to the DG setting.
Finally, targeting mixed form discretizations, the work in \cite{Cheung2022} presents a delicate method that combines the Richardson scheme with an iterative increase in oversample layers to solve the oversampling problem.

\subsection{Dynamic problems}
CEM-GMsFEMs are primarily designed for static operators; however, when applied to dynamic problems, the treatment of time discretization becomes crucial.
In \cite{Li2019}, the authors consider a typical parabolic equation, namely, $${u_t+\mathcal{A}u=f,}$$ using a first-order implicit time discretization.
Quasi-gas dynamics is studied as a model problem in \cite{Chetverushkin2021}, where an additional second-order time derivative term $\alpha u_{tt}$ is incorporated into the parabolic equation.
In both cases, the multiscale space is constructed from the elliptic operator $\mathcal{A}$ and is incorporated into the solver through the variational formulation for time discretization.
In \cite{Chung2020a}, the wave equation is reformulated as a first-order system, and the multiscale space naturally arises from the mixed formulation of the elliptic operator.
Note that even for explicit time discretization of wave equations, one still needs to solve linear systems associated with the mass matrix.
In \cite{Cheung2021}, an IPDG framework combining Petrov--Galerkin schemes is proposed for the wave equation, where the test space is constructed based on the eigenspace while the trial space is obtained from the oversampling procedure.
To avoid the inversion of the mass matrix, the right-hand side of the eigenvalue problem is taken as the mass bilinear form rather than the weighted $L^2$ form derived from the coefficient field.
The framework presented in \cite{Chung2020a} is then extended to parabolic equations in \cite{Wang2021}.
The technique, known as the partially explicit approach, is further formulated in \cite{Chung2021b,Chung2021c,Chung2022}.
The motivation behind this approach is that even when employing a multiscale space from CEM-GMsFEMs, denoted by $V_{H,1}$, for implicit time discretization, the accuracy may still deteriorate over long-time simulations.
To enhance accuracy, an enrichment $V_{H,2}$ is introduced to the multiscale space $V_{H,1}$, with dynamic updates for this component performed explicitly.
We emphasize again that, even for the explicit scheme, linear systems still need to be solved, and the dimension reduction on $V_{H,2}$ is achieved using the mass bilinear form.
Exponential integrators also play an important role in time discretization, and their integration with CEM-GMsFEMs is discussed in \cite{Poveda2024}.
{CEM-GMsFEM has recently been extended to Schr\"{o}dinger equations in \cite{Jin2025}, where the local eigenvalue problem is modified to incorporate the heterogeneity arising from the potential term. In particular, the multiscale basis construction is adapted to the Schr\"{o}dinger setting by employing a Hamiltonian-weighted spectral problem, which allows the method to capture the effects of highly oscillatory and high-contrast potentials.
To address the time-dependent problem, a Crank–Nicolson CEM-GMsFEM scheme was established, providing a stable and second-order accurate temporal discretization. Rigorous analysis demonstrates global convergence of the method, including first-order convergence in the energy norm and second-order convergence in the \( L^2 \) norm. Furthermore, the constructed multiscale basis functions exhibit exponential decay with respect to the oversampling size, ensuring localization and computational efficiency. Under suitable resolution conditions relating the coarse mesh size, spectral threshold, time step size, and semiclassical parameter, full space–time convergence is obtained. Numerical experiments in one and two dimensions with high-contrast and multiscale potentials validate the theoretical findings and demonstrate the robustness and efficiency of the proposed approach.} 
Dual continuum models are widely used in porous flow simulations, where the effects of fractures must be carefully addressed. In these models, an additional physical quantity representing flow in fractures is incorporated into the system.
Both fields exist over the entire domain and are coupled via a lower-order transfer term. In \cite{Cheung2018}, the authors proposed a CEM-GMsFEM for dual continuum models in which the eigenvalue problems are modified to account for the lower-order transfer term.
The problem still exhibits a coercivity property, so the remaining procedures can be naturally applied.
Publications \cite{Poveda2023,Poveda2023a} consider nonlinear compressible single-phase flows, where the nonlinearity arises from conductivity that depends on pressure.
To avoid updating the multiscale bases at every time step, the authors construct the multiscale space using the initial pressure.

\subsection{Non-local multi-continua upscaling}
One beautiful fruit grown from CEM-GMsFEMs is so-called non-local multi-continua (NLMC) upscaling.
This theory has its roots in the modeling of subsurface flow, where fractures are embedded in a porous medium. {For the model problem \cref{eq:model} in fractured media, the computational domain $D$ is decomposed into 
matrix and fracture regions:
\[
D = D_m  \bigoplus_{i=1}^{N_f}d_i D_{f,i}.
\]
where $D_m$ denotes the matrix region and $D_{f,i}$ represents the 
$i$-th fracture region. The quantity $d_i$ denotes the aperture (thickness) 
of the fracture $D_{f,i}$.The permeability field $\kappa(x)$ is defined piecewise by
\[
\kappa(x) =
\begin{cases}
\kappa_m(x), & x \in D_m, \\
\kappa_i(x), & x \in D_{f,i}, \quad i=1,\dots,N_f.
\end{cases}
\]}
Since the conductivity of fractures is significantly higher than that of the matrix, the flow in the fractures is much faster, implying that interactions within fracture networks are inherently non-local in contrast to the local interactions in the matrix.
Given that CEM-GMsFEMs are robust against contrast, it is natural to employ them to model such non-local interactions.
Moreover, CEM-GMsFEMs can provide an upscaling model \cite{Chung2018a} derived from the final linear algebraic system, where oversampling generates a denser matrix that elegantly represents nonlocality.
NLMC for nonlinear flow problems is discussed in \cite{Leung2019}, and as a simplified version of CEM-GMsFEM, the NLMC upscaling approach is also a computational method whose numerical convergence is studied in \cite{Zhao2020}.
The coupling of elasticity is considered in \cite{Vasilyeva2019}, leading to an NLMC upscaling for poroelasticity, while its extension to the dual continuum setting is discussed in \cite{Vasilyeva2019a}.
In \cite{Vasilyeva2019b}, the authors study perforated media where non-homogeneous boundary conditions are imposed on the inner boundaries of perforations.
For nonlinear problems, where linear systems must be solved repeatedly, \cite{Vasilyeva2020} proposes an acceleration technique using machine learning.
By incorporating the temporal dimension into the framework, a space-time NLMC upscaling is proposed in \cite{Hu2025}.
NLMC upscaling is first connected with representative volume elements in \cite{Chung2021}, the connection between homogenization theory and NLMC is further elucidated in \cite{Chung2024}.
In fact, traditional homogenization theory or FE\textsuperscript{2} (HMM) methods also involve solving several local problems subject to certain restrictions, which can be identified in the corresponding components of NLMC upscaling.

\subsection{Indefinite problems}
The original CEM-GMsFEM framework is designed for positive definite problems.
When the problem is indefinite---either non-Hermitian or non-positive definite--a careful treatment is required.
In particular, the eigenvalue problem must be reformulated, as eigenvector expansion for indefinite operators is highly non-trivial.
{For the convection-diffusion equations,
\begin{equation}\label{eq:con}
-\text{div} \big( \kappa \nabla u \big)
+ \textbf{b} \cdot \nabla u
= f
\quad \text{in } D,
\end{equation}
subject to the homogeneous Dirichlet boundary condition $u = 0$ on $\partial D$.} The indefiniteness arises from the convection term \(\bm{b}\cdot \nabla u\) in \cref{eq:con}, which can cause stability issues when it dominates the diffusion term.
In \cite{Chung2020}, the authors, based on the discontinuous Petrov--Galerkin framework, proposed a CEM-GMsFEM stabilization technique.
Here, the eigenvalue problem is associated with the convection field \(\bm{b}\) and is symmetrized, leading that all eigenvalues are real-valued.
The Galerkin discretization with CEM-GMsFEMs for convection-diffusion problems is established in \cite{Zhao2023}, and an extension to inhomogeneous boundary conditions is presented in \cite{Wong2024}.
For Stokes flow problems, which exhibit a saddle-point structure, the convergence of CEM-GMsFEMs is provided in \cite{Chung2021a}; in this case, the eigenvalue problem arises from the vector Laplace operator and the oversampling problem is solved in the divergence-free space. {For Helmholtz problems with the first-order absorbing boundary condition,
find $u : D \to \mathbb{C}$ such that
\begin{equation}\label{eq:helm}
\begin{cases}
- \text{div} (\kappa \nabla u) - k^2 u = f 
& \text{in } D, \\[4pt]
\kappa \nabla u \cdot \vec{n} - i k u = 0 
& \text{on } \partial D.
\end{cases}
\end{equation}
the wave number term $k$ introduces indefiniteness, and the boundary condition renders the operator non-Hermitian.} 
{The CEM-GMsFEM for Helmholtz problems \cref{eq:helm} was proposed in \cite{Jin2024}, where the local eigenvalue problem is derived from an energy functional incorporating the wave number. The central idea is to construct multiscale basis functions that capture the oscillatory behavior of high-frequency waves through carefully designed local spectral problems and constraint energy minimization. A Petrov–Galerkin discretization is employed for the multiscale formulation, in which both trial and test basis functions are constructed via oversampling. This strategy enhances stability for indefinite Helmholtz problems and improves approximation properties by incorporating additional local information beyond each coarse element. Inherited from the CEM-GMsFEM framework, the proposed method achieves a significant reduction in CPU time and DOFs compared to the standard finite element method, while maintaining comparable accuracy. Moreover, the method has been successfully applied to irregular domains and more complex geometries. These additional numerical experiments further confirm the robustness and efficiency of the approach beyond standard rectangular benchmark problems.} Another interesting challenge arises when the coefficient can be both positive and negative,referred to as sign-changing problems, which find important applications in metamaterial simulations. {In \cite{Ye2024a,chung2025multiscale}, the local eigenvalue problem is again formulated based on the energy bilinear form, but with the coefficient replaced by its absolute value in \cref{eq:eigenvalue}.  By using the absolute value of the coefficient in the eigenvalue problem, one obtains a stable auxiliary space that captures the essential multiscale features of the medium. Under suitable technical assumptions and by employing the \texttt{T}-coercivity theory \cite{Bon2010}, the inf-sup stability of the method is rigorously established, together with corresponding a priori error estimates. In contrast, the subsequent oversampling problem is solved using the original sign-changing coefficient. This design allows the constructed multiscale basis functions to retain the correct physical properties of the underlying problem while benefiting from the stability provided by the auxiliary spectral problem. The combination of a coercive eigenvalue formulation and a physically consistent oversampling stage ensures both robustness and accuracy for problems with sign-changing coefficients. }
{\subsection{H(curl)-problems}
The CEM-GMsFEM framework has recently been extended to Maxwell-type problems involving the double-curl operator, which are known for their non-elliptic structure and lack of coercivity. The problem is to find the electric field $\bm{u} \in \mathbb{C}^3$ corresponding to a given current density $\mathbf{f}$ such that the following system holds:
\begin{equation}\label{eq:max}
\begin{cases}
\operatorname{curl}\left(\kappa \operatorname{curl} \mathbf{u}\right) 
- k^2 \mathbf{u} = \mathbf{f}
& \text{in } D, \\[6pt]
\kappa \operatorname{curl} \mathbf{u} \times \mathbf{n} 
- i k \mathbf{u}_T = \mathbf{g}
& \text{on } \partial D.
\end{cases}
\end{equation}
where $k  > 0$ is the (fixed) free-space wavenumber and $\mathbf{u}_T$  denotes the tangential component trace. These intrinsic difficulties often lead to severe numerical instabilities, particularly in high-contrast media and at large wave numbers. In \cite{zhong2026multiscale}, a novel multiscale formulation based on CEM-GMsFEM is proposed for \cref{eq:max} in which the local eigenvalue problems incorporate both a mass term and a Silver–Müller-type boundary penalty. This design restores coercivity at the spectral level and automatically excludes the kernel of the curl operator, thereby eliminating the need to explicitly impose divergence-free constraints on the multiscale basis functions. The multiscale space is then constructed within a Petrov–Galerkin framework using a modified bilinear form. By establishing key norm equivalences under an appropriate resolution condition, the authors rigorously prove the coercivity of the formulation and derive corresponding convergence results. The analysis demonstrates that, provided sufficient oversampling is employed, the method attains first-order convergence independent of local material contrast, which is essential for accurately modeling electromagnetic wave propagation in heterogeneous media \cite{Henning2020max}.  Notably, the design of the local spectral problems represents a substantive departure from existing approaches and addresses challenges that have not been fully resolved in earlier works \cite{ Verfuerth2019, Henning2020max}.}
{
\subsection{Algebraic solver}
The current CEM-GMsFEM framework is closely tied to the underlying numerical discretization and geometric mesh structure, which may limit its broader applicability, particularly in settings where no natural geometric description is available. Developing a fully algebraic framework is therefore desirable, ideally allowing multiscale components to be generated automatically from the given system matrix. In \cite{zhou2026multiscale}, inspired by the CEM-GMsFEM philosophy, the authors propose a purely algebraic multiscale method for highly heterogeneous spatial network models that avoids the introduction of geometric parameters such as mesh sizes or Dirichlet nodes. The authors defined a special network $\mathcal{G}=(\mathcal{N},\mathcal{E})$ where $\mathcal{N}$ and $\mathcal{E}$ denote the sets of nodes and edges based on the \cref{eq:model}. Two symmetric positive semi-definite matrices are introduced: 
$L$, representing a weighted graph Laplacian with possibly high-contrast coefficients, 
and $M$, representing mass properties. Solving \cref{eq:model} is equivalent to seeking $u$ in the space $V$, 
consisting of all real-valued functions defined on the nodes $\mathcal{N}$, 
such that
\begin{equation*}
\big( (L + M) u , v \big)
=
(f, v)
\quad \forall v \in V,
\end{equation*}
where $(\cdot,\cdot)$ denotes the Euclidean inner product on $V$. The central idea is to construct localized multiscale basis functions via graph-based spectral decompositions combined with oversampling layers, together with a subgraph-wise estimate that defines a geometry-independent Poincaré constant. This formulation achieves convergence independent of local heterogeneity contrast while preserving the intrinsic algebraic structure of the network system. Rigorous theoretical analysis establishes stability and contrast-robust error estimates, and numerical experiments demonstrate substantial reductions in degrees of freedom without sacrificing accuracy. Extension to a learning-accelerated algebraic multilevel preconditioner can be found in \cite{LIU2026114950}}.
\subsection{Other applications}
Actually, extending the original CEM-GMsFEMs to inhomogeneous boundary conditions is not straightforward, as the convergence of CEM-GMsFEMs relies on the \(L^2\) right-hand source term.
For generalized boundary conditions, we cannot transform the problem into a variational form with an \(L^2\) right-hand side.
In \cite{Ye2023a}, the authors proposed several boundary correctors to mitigate this issue, and the computations for these correctors can also be localized.
The framework is further extended to elasticity in \cite{Wang2024}.
For nonlinear boundary conditions involving variational inequalities, the article \cite{Li2025} introduces an iterative scheme in combination with semi-smooth Newton methods.
Finally, Xie et al.\ \cite{Xie2024} presented an application of CEM-GMsFEMs to perforated domains, featuring numerical experiments on unstructured meshes.
{\section{Numerical illustration of the decay property}
\label{sec:num}}
The effectiveness of the CEM-GMsFEM relies on the rapid decay of the multiscale basis functions with respect to \( l \), the number of oversampling layers, as established in \cref{lem:step3}. It is necessary to numerically visualize the exponential decay behavior of the multiscale basis functions obtained in \cref{sec:keys}. A square domain is considered with a \(10 \times 10\) periodic high-contrast coefficient \( \kappa \) in \cref{eq:model}, as shown in subplot (a) of \cref{fig:eigen-0d1+1d0}. In each periodic cell, a centered square inclusion is assigned the coefficient value \( \kappa = 1 \), whereas the background coefficient is \( \kappa = 10^3 \). The coarse mesh $\mathcal{T}_H$ is aligned with the periodic structure. The numerical experiments are performed on the square domain $\Omega = (0,1) \times (0,1)$ with mesh sizes $(H,h) = (1/10, 1/400)$, and $k = 3$ eigenfunctions are used in each coarse subdomain $K_j\in\mathcal{T}_H$.

The selected coarse element highlighted in red in subplot (a) of  \cref{fig:eigen-0d1+1d0} and plot the first three eigenfunctions \( \Psi_{j,1} \), \( \Psi_{j,2} \), and \( \Psi_{j,3} \), computed from \cref{{eq:eigenvalue}}, in subplots (b)–(d). Due to the use of homogeneous Neumann boundary conditions, the first eigenvalue is always zero, and the corresponding eigenfunction represents the constant mode associated with the kernel of the local operator. The higher-order eigenfunctions exhibit localized oscillatory structures, illustrating the spectral separation that underpins the rapid decay of the oversampled multiscale basis functions.

\begin{figure}[!ht]
  \includegraphics[width=\textwidth]{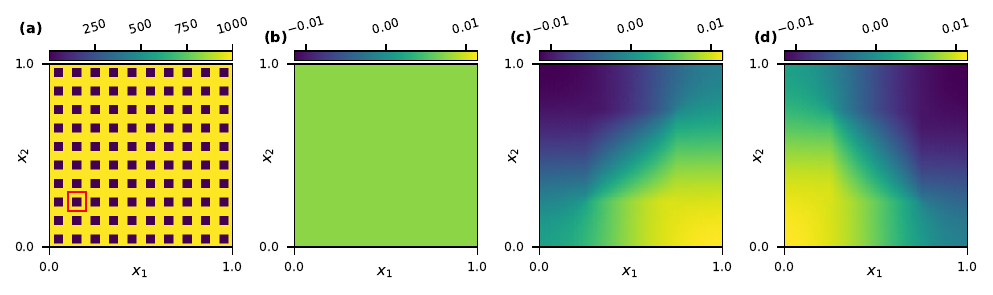}
  \caption{
    \SubplotTag{(a)} The coefficient profile and the marked coarse element.
    \SubplotTag{(b)}--\SubplotTag{(d)} The plot of the first/second/third eigenfunction corresponding to the marked coarse element.
  }\label{fig:eigen-0d1+1d0}
\end{figure}
The decay of the multiscale basis functions \( \Phi_{j,1} \), \( \Phi_{j,2} \), and \( \Phi_{j,3} \), obtained from \cref{eq:G-j-l} with different numbers of oversampling layers, is illustrated in \cref{fig:ms-0d1+1d0}. In this figure, the first, second, and third rows correspond to the first, second, and third eigenfunctions, respectively, while the first, second, and third columns display the multiscale basis functions computed with \( l = 1 \), \( 2 \), and \( 3 \), respectively. The location of the selected coarse element determines the maximum admissible number of oversampling layers, which is \( l_{\max} = 8 \) in this case. We therefore compute the relative differences in both the energy norm and the \( L^2 \)-norm of the multiscale basis functions between the reference solution with \( l = 8 \) and those with \( l = 1, \dots, 7 \). The results are shown in the fourth column of \cref{fig:ms-0d1+1d0}. From the plots of the multiscale basis functions, we observe that they decay rapidly away from the selected coarse element. Consequently, accurate approximations can be achieved using only a small number of oversampling layers \cite{Chung2018}.
\begin{figure}[!ht]
  \includegraphics[width=\textwidth]{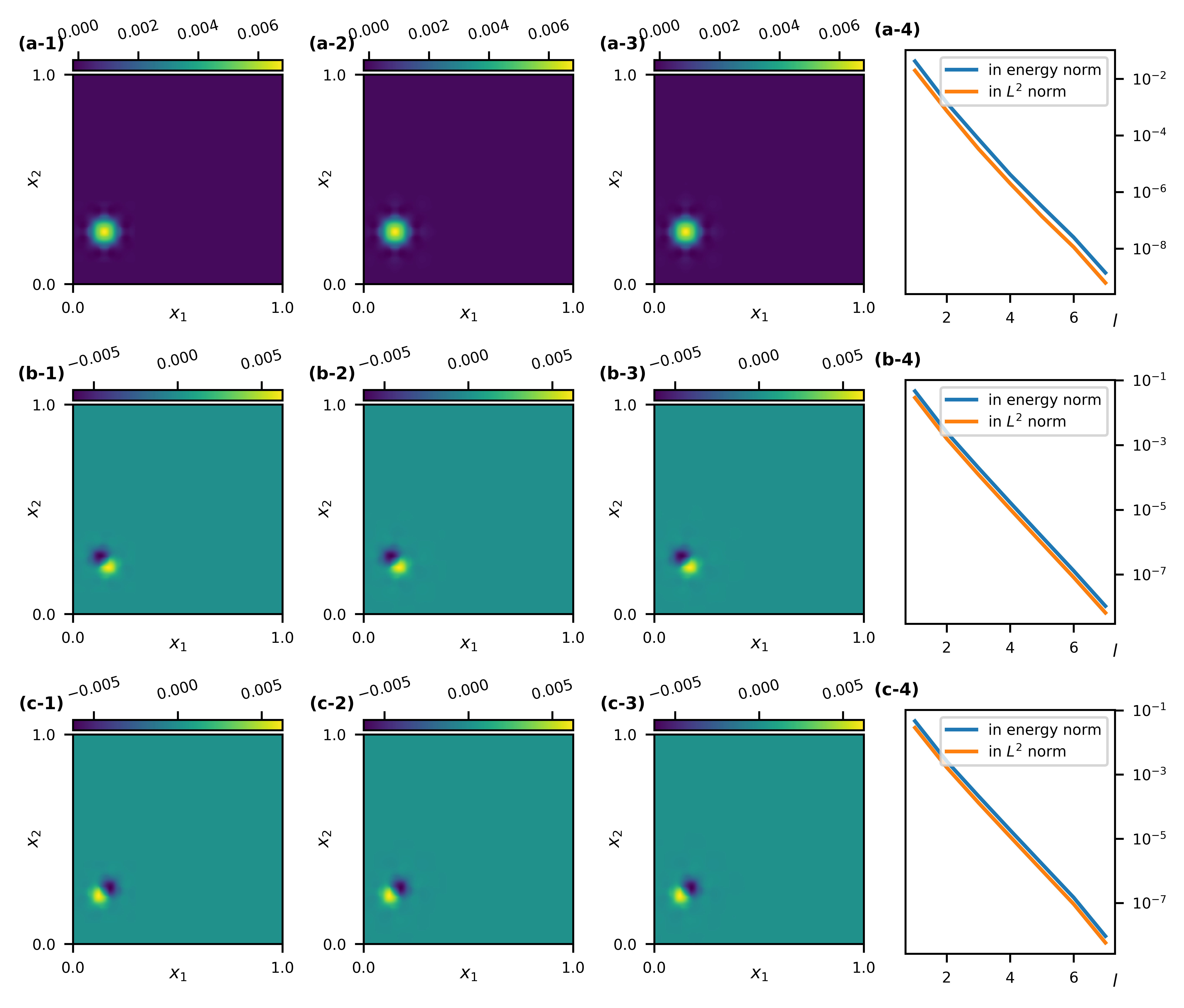}
  \caption{
    The subplots are marked as \SubplotTag{(x-y)}, where \SubplotTag{x} can take \SubplotTag{a}, \SubplotTag{b}, or \SubplotTag{c}, corresponding to the results for the first, second, or third eigenfunction, respectively.
    If \SubplotTag{y} is \SubplotTag{1}, \SubplotTag{2}, or \SubplotTag{3}, the subplot displays the multiscale basis with $l$ oversampling layers, $l$ equal to \SubplotTag{y}.
    Alternatively, if \SubplotTag{y} is \SubplotTag{4}, the subplot shows the relative differences (y-axis) in the energy and $L^2$ norm of the multiscale bases between $l=8$ and $l=1,\dots,7$ (x-axis).
  }\label{fig:ms-0d1+1d0}
\end{figure}
\section{Conclusion and outlook}
\label{sec:conclusion}
We have reviewed the CEM-GMsFEM framework, a powerful tool for solving elliptic problems with highly heterogeneous coefficients.
We have discussed the construction of multiscale spaces, the convergence properties of CEM-GMsFEMs, and the error estimates for localized multiscale spaces.
Additionally, we have presented recent applications of CEM-GMsFEMs, including elasticity problems, alternative discretizations, dynamic problems, non-local multi-continua upscaling, indefinite problems, and more.

We believe the following topics merit further investigation:
\begin{itemize}
  \item \textbf{Computing efficiency}: Constructing multiscale bases (the offline stage) is computationally expensive.
        Although the potential for parallelization is well recognized, practical implementations remain scarce and challenging.
        Furthermore, the resulting linear system is much denser than that arising from standard finite element discretizations, and efficient solvers beyond direct methods remain an open research question.

 \item {\textbf{Neural operator}: Existing neural operator models often exhibit performance degradation in strongly heterogeneous or high-contrast settings. In \cite{rudikov2026locally}, a GMsFEM-based neural operator (GMsFEM-NO) was proposed to enhance robustness in multiscale problems. Since the current CEM-GMsFEM framework inherits the stability and contrast-robust properties of GMsFEM, it provides a natural foundation for integration with operator learning techniques. This combination offers a promising direction for developing reliable neural operators tailored to high-contrast and nonlinear multiscale problems.}


  \item \textbf{Homogenization and non-locality}: Although NLMC upscaling provides insight into how non-locality arises from multiscale methods, its physical interpretation remains unclear.
        Investigating this phenomenon from the perspective of \emph{generalized continuum theories} is worthwhile. Mechanical examples, particularly those related to fiber-reinforced composites and mechanical lattices, may serve as ideal test cases for such studies.

  \item \textbf{Nonlinear model reduction}: The current rigorous error estimates rely on the linearity of the problem.
        For nonlinear problems, several stringent assumptions are required, and updating multiscale bases at each iteration incurs high computational costs.
        Developing efficient strategies for nonlinear model reduction is a common challenge across science and engineering and represents an exciting direction for further research within the CEM-GMsFEM framework.
\end{itemize}

\section*{Acknowledgments}
EC's research is partially supported by the Hong Kong RGC General Research Fund (Project numbers: 14305624 and 14304525). The authors would like to thank the kind hospitality of the Institute for Mathematical and Statistical Innovation during the workshop Reduced-Order Modeling for Complex Engineering Problems: From Analysis to Practical Implementation.

\bibliographystyle{siamplain}
\bibliography{refs.bib}

\end{document}